%% file: main.tex
\documentclass[a4paper,12pt]{amsart}
\usepackage[T1, T2A]{fontenc}
\usepackage[utf8]{inputenc}
\usepackage[english]{babel}

\usepackage{amsmath}
\usepackage{amssymb}
\usepackage{amscd}
\usepackage{amsthm}
\usepackage{concrete}
\usepackage{float}

\usepackage[all]{xy}
\usepackage{graphicx}
\usepackage{hyperref}
\usepackage[shortlabels]{enumitem}
\usepackage{tikz}
\usetikzlibrary{decorations.pathreplacing}

\newtheorem{proposition}{Proposition}
\newtheorem{conjecture}{Conjecture}
\newtheorem*{conjecture*}{Conjecture}
\newtheorem{question}{Question}

\newtheorem{lemma}{Lemma}
\newtheorem{theorem}{Theorem}
\newtheorem*{theorem*}{Theorem}

\theoremstyle{definition}
\newtheorem{definition}{Definition}
\newtheorem*{definition*}{Definition}
\newtheorem{example}{Example}
\theoremstyle{remark}

\newcommand{\ul}[1]{\underline{#1}}
\newcommand{\ol}[1]{\overline{#1}}
\newcommand{\mc}[1]{\mathcal{#1}}
\newcommand{\wt}[1]{\widetilde{#1}}

\def\K{\mathbb{K}}
\def\indlim{\lim_{\longrightarrow}}
\def\Z{\mathbb{Z}}
\def\QQ{\mathbb{Q}}

\def\A{\mathbb{A}}
\def\Ga{\mathbb{G}_a}
\def\Gm{\mathbb{G}_m}

\def\la{\langle}
\def\ra{\rangle}
\def\ss{\subset}
\def\sse{\subseteq}
\def\T{\mathbb{T}}

\def\L{\mathcal{L}} 
\def\vk{\varkappa}
\def\pd{\partial}
\def\G{\mc{G}}
\def\ve{\varepsilon}

\DeclareMathOperator{\Aut}{Aut}
\DeclareMathOperator{\SAut}{SAut}
\DeclareMathOperator{\GCD}{gcd}
\DeclareMathOperator{\Ad}{Ad}
\DeclareMathOperator{\Der}{Der}
\DeclareMathOperator{\ad}{ad}

\DeclareMathOperator{\LND}{LND}

\begin{document}
		\title[Infinite transitivity for automorphism groups of the affine plane]{Infinite transitivity for automorphism groups of the~affine plane}
	\author{Alisa Chistopolskaya}
	\author{Gregory Taroyan}
	\address{Higher School of Economics, Faculty of Mathematics, Usacheva street, 6, Moscow, 119048 Russia}
	\email{achistopolskaya@gmail.com}
	\address{Higher School of Economics, Faculty of Mathematics, Usacheva street, 6, Moscow, 119048 Russia}
	\email{tgv628@yahoo.com}
	
	
	\thanks{The first author was supported by the grant RSF 19-11-00172.}
	\begin{abstract}
	This paper is dedicated to the problem of infinite transitivity for algebraically generated automorphism groups of the affine plane. We provide a necessary and sufficient condition of infinite transitivity for a large family of subgroups generated by one-parameter additive groups.
	\end{abstract}
	\maketitle
\section{Introduction}
 
This paper is devoted to the study of the automorphism group $\Aut(\A^2)$ of the affine plane~$\A^2$ over an algebraically closed field $\K$ of characteristic zero. 
One approach to the study of this group is to look for "elementary" subgroups in it. The first type of groups that comes to mind is that of \(\Gm\)-groups. By definition \(\Gm\)-groups are subgroups of the group~\({\Aut(\A^2)}\) that are isomorphic to the algebraic group of invertible elements in the ground field. Actions of these groups are very important and very well studied. The second type of nice subgroups one may think about is that of \(\Ga\)-groups. That is subgroups in \(\Aut(\A^2)\) that are isomorphic to the additive group of the ground field. Although seemingly even more simple than \(\Gm\)-groups, \(\Ga\)-groups and their actions turn out to be far more complicated.
Below we describe two families \(\Ga\)-actions on the plane~\(\A^2\). These two families are the main subjects of study in this work.
\par For any $a,b\in\Z_{\ge 0}$, define the following one-parameter subgroups of $\Aut(\A^2)$: 
$$H_a\colon(x,y)\mapsto (x+\alpha y^a,y)\;\;\mbox{and}\;\; K_b\colon(x,y)\mapsto (x, y+\beta x^b),\;\; \alpha , \beta \in\K\,.$$
\par Let \(\SAut(\A^2)\) be the subgroup of \(\Aut(\A^2)\) generated by all one-parameter subgroup. Importance of families \(H_*,K_*\) now stems from the fact that they generate the whole group \(\SAut(\A^2)\) of the plane by Jung--Van der Kulk theorem.
\par
The importance of transitive group actions in geometry is well understood. One may ask when an action of a group is not only transitive, but is "very" transitive. In more concrete terms an action of a group \(G\) on a set \(S\) is called \(m\)-transitive for some positive integer \(m\) if the diagonal action on the configuration space of \(m\) pairwise distinct ordered points in \(S\) is transitive. In these terms transitive actions become \(1\)-transitive actions. Now it is natural to define "very" transitive actions to be those which are \(m\)-transitive for an arbitrary \(m\ge 1.\) Such actions are called {\it infinitely transitive}. 
\par In $2018$, D.Lewis, K.Perry, and A.Straub proved \cite{LPS} that the subgroup of $\Aut(\A^2)$ generated by the pair $H_1, K_2$ acts on $\A^2\setminus\{0\}$ infinitely transitively.
This result is the answer to the question of I.Arzhantsev, K.Kuyumzhiyan, and M.Zaidenberg \cite{AKZ}. It has been proven in \cite{AKZ} that if $ab \ne 2$, the subgroup $G=\langle H_a, K_b\rangle$ of $\Aut(\A^2)$ can not act infinitely transitively. 
We obtain an easier way to prove the result from the work \cite{LPS}, and also generalize these facts to any number of the above one-parameter subgroups.
We consider the group~$G$ generated by a finite number of one-parameter subgroups of the form $H_a$ and $K_b$ and give an equivalent condition to the infinite transitivity of the group $G$ in combinatorial terms. 
\par This paper is organized as follows. In section \ref{sec:prel} we give a brief summary of the results needed. In particular in subsection \ref{subsec:toric} we describe in detail the standard structure of an affine toric variety on the plane $\A^2.$ Subsection \ref{subsec:hom_der} is dedicated to the homogeneous derivations of the coordinate algebra of the affine plane. In subsection \ref{subsec:Dem_roots} we recall combinatorial description of LND in \(\K[x,y].\) In subsection \ref{subsec:two_res} we collect important properties of LNDs in \(\K[x,y].\) Finally in subsection \ref{subsec:ind_aut} we deal with the structure of an Ind-group~$\Aut(\A^2).$ Results of this subsection are at the cornerstones of our work. These results allow us to determine the closure for subgroups of $\Aut(\A^2)$ generated by one-parametric subgroups. Section \ref{sec:main} contains our main results. In section \ref{sec:questions} we give several examples of our results as well as formulate two conjectures generalizing Theorem~\ref{thm2}. Proof of these conjectures is the main goal of our planned subsequent research. Additionally, in the same section we pose a question similar in nature to our results, but for non-affine toric varieties. Answering this question could help develop the theory in a broader context.
\subsection*{Summary of results}
The main results of this paper are the following theorems:
\begin{theorem*}[{\ref{H1}}]
	 Let \(m\) be a positive integer. Then the group $\la H_1,K_{d_1},\ldots,K_{d_m}\ra$ acts on $\A^2\setminus\{0\}$ infinitely transitively if and only if $\GCD(d_1-1,\ldots, d_m-1)=1.$
\end{theorem*}
\begin{theorem*}[{\ref{thm2}}] Let \(m,s\) be positive integers. Then the group $G=\la H_{c_1},\ldots, H_{c_s},K_{d_1},\ldots,K_{d_m}\ra$ acts on $\A^2\setminus\{0\}$ infinitely transitively if and only if
    $$\Z\left\la \begin{pmatrix}-1 & c_1\end{pmatrix},\ldots,\begin{pmatrix}-1 & c_s\end{pmatrix},\begin{pmatrix}d_1& -1\end{pmatrix},\ldots,\begin{pmatrix}d_m & -1\end{pmatrix}\right\ra=\Z^2.$$
\end{theorem*}
   
These theorems generalize results of \cite{LPS} and provide an effective criterion of infinite transitivity for an infinite family of algebraically generated automorphism groups.\hfill\break
\par We are grateful to Ivan Arzhantsev for his support and Sergey Gayfullin for fruitful discussions and valuable comments.
\section{Preliminaries}\label{sec:prel}
\subsection{Toric variety $\A^2$}\label{subsec:toric}

Let $\K$ be an algebraically closed field of characteristic zero and~$\A^2$~be the two-dimensional affine space over $\K$.
Two-dimensional torus~$\T\simeq
(\K^\times)^2$~acts on~$\A^2$~via the
formula
$$
    (t_1,t_2)\circ (x,y)=(t_1 x, t_2 y),\quad t_1,t_2\in \K^\times,\; x,y\in \K.
$$
We call this torus action {\it standard}. Let $M$ be the character lattice of the torus $\T$ and $N$ be the lattice of one-parameter subgroups. The natural pairing $\langle \,\cdot\,,\,\cdot\,\rangle\colon N \times M \to \Z$ extends to the pairing $\langle \,\cdot\,,\,\cdot\,\rangle\colon N_\QQ \times M_\QQ \to \QQ$ between the associated vector spaces $N_\QQ=N\otimes\QQ$ and $M_\QQ=M\otimes\QQ$.
Fix a basis of $N$ and its dual basis of $M$. 
Let $\sigma \sse N_\QQ$ and $\sigma^\vee \sse M_\QQ$ be the positive quadrants. 
Notice that $\sigma$ and $\sigma^\vee$ are dual with respect to natural pairing, i.e. 
$$
 \sigma^\vee=\{u\in M_{\QQ} \mid \langle v,u\rangle\ge  0\,\,\, \forall v\in \sigma\}.
$$
For $m=(m_1,m_2)\in M$ by $\chi^m$ we denote a Laurent monomial $x_1^{m_1} x_2^{m_2}$. 
One can easily see that the standard torus action corresponds to the standard grading of the coordinate ring of $\A^2$, which coincides with the natural $\Z^2$-grading of the polynomial algebra:
$$
\K [x_1, x_2]=\bigoplus_{m\in M\cap\, \sigma^\vee}\K\chi^m.
$$

For the combinatorial description of a general toric affine variety see \cite{Co, Fu}.

\subsection{Homogeneous derivations of the coordinate ring of $\A^2$}\label{subsec:hom_der}

Let us recall that a derivation $\partial \in \Der(\K[x_1, x_2])$ of the polynomial algebra $\K[x_1,x_2]$ is a linear map $\partial\colon \K[x_1,x_2] \to \K[x_1,x_2]$ satisfying the Leibniz rule $\partial(fg) = \partial(f)g + f\partial(g)$ for any $f, g \in \K[x_1,x_2]$. 
\begin{definition}
The derivation is called \emph{homogeneous} with respect to the standard grading discussed in the previous section if the image of any homogeneous polynomial is homogeneous as well. In this case, there is an element $\deg \partial \in \Z^2$ called the \emph{degree} of the derivation~$\partial$ such that $\deg \partial + \deg f = \deg \partial(f)$ for any homogeneous $f \in \K[x_1, x_2]$. 
\end{definition}
\begin{definition}The derivation $\partial$ is called \emph{locally nilpotent} if for any $f \in \K[x_1, x_2]$ there exists a positive integer~$n$ such that $\partial^n(f) = 0$. The set of locally nilpotent derivations is denoted by $\LND(\K[x_1, x_2])$. 
\end{definition}
Let $\partial\in\Der(\K[x_1, x_2])$ be a locally nilpotent derivation. There exists a group $\exp(\K\partial)$ which is {\it a $\Ga$-subgroup} (one-parameter group) of~$\Aut(\A^2).$  Elements of $\exp(\K\partial)$ act on functions from $\K([x_1, x_2])$ via the Baker--Campbell--Hausdorff formula:
\begin{equation}\label{eq:BCH_formula}
    \exp(t\partial)(f)=f+\sum_{n=1}^\infty t^n\frac{\partial^n(f)}{n!},\quad t\in \K, \; f\in \K[x_1, x_2].
\end{equation}

Since $\partial$ is locally nilpotent, for every $f$ this sum has finite number of non-zero summands and the definition is correct. One can see that $\exp$ is one-to-one correspondence between $\Ga$-subgroups of $\Aut(\A^2)$, normalized by standard torus action, and homogeneous locally nilpotent derivations of $\K[x_1, x_2]$. 

\subsection{Demazure roots}\label{subsec:Dem_roots}
Let us denote by $\Xi = \{\rho_1, \rho_2\}$ the set of ray generators of a cone~$\sigma$.
\begin{definition}
A {\it Demazure root} which corresponds to a primitive ray generator $\rho_i \in \Xi$ is a vector $e \in M$ such that the following conditions hold:
\begin{enumerate}
    \item $\langle \rho_i, e \rangle =-1 $;
    \item $\langle \rho_j, e \rangle \ge 0 $ for all $j \neq i$. 
    \end{enumerate}
For more detail, see \cite{De}, \cite{Lie} and \cite{Li}.
\end{definition}

A Demazure root $e$ which corresponds to a primitive ray generator $\rho \in \Xi$ gives a homogeneous locally nilpotent derivation $\partial_{\rho, e}$ of degree $e$ via the formula
$$
\partial_{\rho, e} (\chi^m) = \langle \rho, m \rangle \chi^{m+e}.
$$
All homogeneous $\LND$s arise this way (see \cite[Theorem 2.7]{LiC1}). Thus we have one-to-one correspondence between Demazure roots and $\Ga$-subgroups of $\Aut(\A^2)$ normalized by the torus \(\T.\)

\subsection{Two important properties of LND's in \(\K[x,y]\)}\label{subsec:two_res}
Below we formulate \cite[Lemma 1.10]{Li}, \cite[Lemma 4.7]{AKZ} for the case of the polynomial algebra in two variables.
\begin{lemma} \label{non-homog-deriv}
\begin{itemize}
\item[\rm (a)] Any derivation $\partial\in\Der(\K[x_1, x_2])$ admits a decomposition \begin{equation}\label{hom_dec}
\partial= \sum_{e \,\in\, \sigma^\vee\cap\, M} \partial_e\end{equation} 
where $\partial_e$ is a homogeneous derivation of degree $e$. 
\item[\rm (b)] The set $S(\partial) =\{e \in \sigma^\vee\cap M | \partial_e\neq 0\}$ is finite. Its convex hull $N(\partial)$ 
is a polygon called the {\rm Newton polygon of} $\partial$. 
\item[\rm (c)] If $\partial\in {\LND}\,(\K[x_1, x_2])$  then for any face $\tau$ of $N(\partial)$ one has
$$\partial_\tau:= \sum_{e\in \tau\cap M} \partial_e\in {\LND}\,(\K[x_1, x_2])\,.$$ In particular, for any vertex $e$ of $N(\partial)$ one has $\partial_e\in {\LND}\,(\K[x_1, x_2])$. 
\end{itemize}
\end{lemma}

Every one-parameter subgroup $T$ of the torus $\T$ corresponds to the vector $\rho\in \sigma.$ Moreover, $\rho$ determines a linear form $\L_T$ on the lattice $M$ via the formula 
$$
\L_T(e)=\la e,\rho\ra, \quad e\in M.
$$
\par 
Let us recall that the adjoint action of the Lie algebra $\Der(\K[x_1, x_2])$ on itself is given by the formula
$$\ad_\partial (U)=[\partial,U]=\partial\circ U-U \circ \partial. $$
The following lemma describes the adjoint action of $\Ga$-subgroups, which arise from locally nilpotent derivations $\partial.$

\begin{lemma}\cite{Ma}\label{ad-LND}
If $\pd\in\LND(\K[x_1, x_2])$ then the corresponding adjoint action operator ${\ad}_\pd\in {\rm End}\,(\Der(\K[x_1, x_2]))$ is locally nilpotent.
Moreover, adjoint action of $\exp(\K\partial)$ on $U$ is described by the Baker-Campbell-Hausdorff formula: 
    \begin{gather}\label{BCH_form}
        \Ad_{\exp(\K\partial)}(U)=U+\sum_{n=1}\frac{\ad^n_\partial(U)}{n!}.
    \end{gather}
\end{lemma}

\begin{definition} Let \(\L_T=\alpha\) be a supporting affine line of the Newton polytope \(N(\pd)\). Then we define the {\it T-principal part} $\partial_T$ of the derivation $\partial$ corresponding to this line by
$$\partial_T:= \sum_{e\,\in\,\{\L_T=\alpha\}\cap S(\partial)} \partial_e.$$
\end{definition}

According to Lemma~\ref{non-homog-deriv}, if $\partial$ is locally nilpotent then also $\partial_T$ is.
\subsection{Ind-group structure on $\Aut(\A^2)$}\label{subsec:ind_aut}
The group of regular automorphisms of~$\A^2$ has the structure of an ind-group \cite{Kr, Sasha}. 

\noindent It means that $$G = \indlim \Sigma_s.$$ Here $\Sigma_1 \hookrightarrow \Sigma_2 \hookrightarrow \Sigma_3 \hookrightarrow \ldots$ is a sequence of closed embeddings of algebraic varieties $\Sigma_s$, and for each pair $(i, j)$ there exists a number $n(i, j)$ such that $\Sigma_i \times \Sigma_j \to \Sigma_{n(i, j)}$, $(g, h) \mapsto gh^{-1}$ is a morphism of algebraic varieties. 

The closure of a subgroup $G\sse\Aut(\A^2)$ is defined as
$$
\overline{G} = \indlim \overline{G\cap\Sigma_s}.
$$
\begin{definition}
A group $G$ acts on a set $X$ {\it infinitely transitively}, if it acts $m$-transitively for any positive integer $m$, which means that for any two sets of pairwise distinct points $(x_1,\ldots,x_m)$ and $(y_1,\ldots,y_m)$ there exists $g \in G$ such that $g.x_i=y_i$ for $i=1,\ldots,m$.
\end{definition}

\medskip

The next propositions are crucial for our understanding of subgroups closure of~$\Aut(\A^2)$.
\begin{proposition}\label{principal-part}\cite[Proposition 4.16]{AKZ} Consider a subgroup $G\sse \Aut(X)$ 
normalized by a one-parameter subgroup $T$ of the torus $\T$. Let $H=\exp(\K\partial)$ be a $\Ga$-subgroup 
of $G$ where $\partial\in {\LND}\,(\K[x_1, x_2])$, and let $\partial_T\in {\LND}\,(\K[x_1, x_2])$ be the $T$-principal part of $\partial$. 
Then $H_T=\exp(\K\partial_T)$ is a $\Ga$-subgroup of $\overline{G}$. \end{proposition}
\par Note that any algebraic group is an ind-group. Indeed, let \(H\) be an algebraic group, then sequence of \(\Sigma_i=H\) for all \(i\ge 1\) defines a structure of an ind-group on \(H.\)
\begin{definition}
	A subgroup \(H\) of an ind-group \(\G\) is called {\it algebraic} if there exists a structure of an algebraic group on \(H\) such that \(H\hookrightarrow\G\) is a morphism of ind-groups.
\end{definition}
\begin{definition}
    A subgroup $G\sse\Aut(X)$ generated by a family of connected algebraic subgroups of $\Aut(X)$ 
is called \emph{algebraically generated}.
\end{definition} Orbits of an algebraically generated automorphism group $G$ are locally closed subsets of $X$ 
in the Zariski topology \cite[Propsition 1.3]{AFKKZ}.  

\begin{proposition}\label{lem:many-points}\cite[Proposition 3.4]{AKZ}
Let $G\sse \Aut(X)$ be an algebraically generated subgroup. Then the following holds.
\begin{itemize}
\item[{\rm (a)}] The orbits of $G$ and of $\overline{G}$ in $X$ are the same. In particular, if \(\ol{G}\) acts on \(X\) with an open orbit \(\mc{O}_{\ol{G}}\) then also \(G\) does and \(\mc{O}_G=\mc{O}_{\ol{G}}.\)
\item[{\rm (b)}]  If  $\overline{G}$ acts $m$-transitively on ${{\mathcal{O}}_{\overline{G}}}$ then $G$ does also.
\item[{\rm (c)}]  If  $\overline{G}$ acts  infinitely transitively on ${{\mathcal{O}}_{\overline{G}}}$ then also $G$ does.
\end{itemize}
\end{proposition}

\section{Main Results}\label{sec:main}
Let us introduce two special types of $\Ga$-subgroups of automorphism group $\Aut(\A^2).$ 

\noindent For any $a,b\in\Z_{\ge 0}$
$$H_a\colon(x,y)\mapsto (x+\alpha y^a,y)\;\;\mbox{and}\;\; K_b\colon(x,y)\mapsto (x, y+\beta x^b),\;\; \alpha , \beta \in\K\,.$$
    
In fact, these groups and groups which arise from homogeneous locally nilpotent derivations are the same: one can see that $H_a = \exp(\K y^{a}\frac{\partial}{\partial x})$ and $K_b = \exp(\K x^{b}\frac{\partial}{\partial y})$.
Let us also denote by $K_s^\infty$ the group generated by $K_b$ for $b \geqslant s\,$: $K_s^\infty = \la K_{s}, K_{s+1}, \ldots,K_{s+n},\ldots\ra$. Notice that $K_s^\infty$ consists of all automorphisms mapping $(x,y)$ to $(x, y+S(x)),$ where $S(x)$ is divisible by $x^s.$ 

\begin{lemma}\label{secondPart}
Let $G = \la H_t, K_s^\infty \ra$ for some $s,t \in \Z_{>0}$. Then $G$ acts on $\A^2\setminus\{0\}$ infinitely transitively.
\end{lemma}
\begin{proof}
For a point P, we denote by $x(P)$ the $x$-coordinate of $P$ and by $y(P)$ the $y$-coordinate of $P$. Let us show that for any distinct points $P_1,\ldots,P_m, P, Q \in \A^2\setminus\{0\}$ we can find an element $g \in G$ such that $g.P_i = P_i$ for all $i=1,\ldots,m$ and $g.P = Q$. 

\smallskip

{\bf Step $1$.} Let us construct an element $g_1 \in G$ such that 
\begin{equation}\label{step1}
x(g_1.P_i) \neq x(g_1.P) \,\text{ for all }\, i=1,\ldots,m \;\text{ and }\; x(g_1.P) \neq 0.
\end{equation}
Since the ground field $\K$ is infinite and the points $P_i$ are distinct from the point $P$, we can find an element $\alpha \in \K$ such that $x(P) + \alpha y(P)^t \neq 0$ and the inequality
$$\alpha (y(P_i)^t - y(P)^t) \neq x(P)-x(P_i) \quad \Longleftrightarrow \quad
x(P_i) + \alpha y(P_i)^t \neq x(P) + \alpha y(P)^t$$
holds for all $i=1,\ldots,m$. Then an element $g_1 \in H_t$, $g_t\colon(x,y)\mapsto (x+\alpha y^t,y)$ satisfies the condition~\eqref{step1}. 

\smallskip

{\bf Step $2$.} Let us construct an element $g_2 \in G$ such that 
\begin{equation}\label{step2}
g_2g_1.P_i=g_1.P_i \,\text{ for all }\, i=1,\ldots,m \;\text{ and }\; y(g_2g_1.P) = y(g_1.Q).
\end{equation}
By means of Hermite interpolation, we can construct a polynomial $R(t)$ with coefficients in $\K$ such that following conditions hold:
\begin{enumerate}
    \item[(a)] $t^s$ divides $R(t)$;
    \item[(b)] $R(x(g_1.P_i)) = 0$ for all $i=1,\ldots,m$;
    \item[(c)] $R(x(g_1.P)) = y(g_1.Q)-y(g_1.P)$.
\end{enumerate}
Indeed, by step 1 the numbers $x(g_1.P_i)$ differ from $x(g_1.P)$ and $x(g_1.P)$ differs from zero, so there is a polynomial $R$ which satisfies $R(0) = R'(0) = \ldots = R^{(s-1)}(0) = 0$ and the conditions $(b)$, $(c)$. 

We take an element $g_2\colon(x,y)\mapsto (x,y + R(x))$, then $g_2 \in K_s^\infty$ according to the condition~$(a)$. From the the condition~$(c)$ it follows that $y(g_2g_1.P) = y(g_1.Q)$, and because of the condition~$(b)$ the equality $g_2g_1.P_i=g_1.P_i$ holds for all $i=1,\ldots,m$. 

\smallskip

Notice that for constructed elements $g_1, g_2 \in G$ we have
\begin{equation}
    y(g_1^{-1}g_2g_1.P) = y(Q) \;\text{ and }\; g_1^{-1}g_2g_1.P_i=P_i.
\end{equation}
Indeed, $y(g_1^{-1}g_2g_1.P)= y(g_2g_1.P)=y(g_1.Q)= y(Q)$. 

\smallskip

{\bf Step $3$.} Let us denote $P_0 =g_1^{-1}g_2g_1.P$. Similarly to the step $1$ we choose $g_3 \in H_t$ such that inequalities
$x(P_i) \neq x(P_j) \neq x(Q) \neq 0$ hold for all $i, j = 0,\ldots,m$.
Let $S(t)$ be an element of the polynomial ring $\K[t]$ such that following conditions hold:
\begin{enumerate}
    \item[(a)] $t^s$ divides $S(t)$;
    \item[(b)] $y(g_3.P_i)+ S(x(g_3.P_i)) = 0$ for all $i=1,\ldots,m$;
    \item[(c)] $y(g_3.P_0)+ S(x(g_3.P_0)) = y(g_3.Q)+ S(x(g_3.Q)) \neq 0$.
\end{enumerate}
We take an element $g_4 \in K_s^\infty$, $g_4\colon(x,y)\mapsto (x,y + S(x))$. Because of the conditions above, $y(g_4g_3.P_i)=0$ for all $i=1,\ldots,m$ and $y(g_4g_3.P_0) = y(g_4g_3.Q) \neq 0$.

\smallskip

{\bf Step $4$.} Now we have $y(g_4g_3.P_0) = y(g_4g_3.Q) \neq 0$ and $y(g_4g_3.P_i)=0$ for all $i=1,\ldots,m$. There is an element $g_5 \in H_t$, that maps $g_4g_3.P_0$ to $g_4g_3.Q$. Indeed, if $$\beta =\frac{x(g_4g_3.Q)-x(g_4g_3.P_0)}{y(g_4g_3.P_0)^t},$$
then $g_5\colon(x,y)\mapsto (x + \beta y^t,y)$ maps $g_4g_3.P_0$ to $g_4g_3.Q$. And since for all $i$ holds $y(g_4g_3.P_i)=0$, we have $g_5g_4g_3.P_i = g_4g_3.P_i$ for all $i=1,\ldots,m$.

\smallskip

{\bf Step $5$.}
We take an element $g = g_3^{-1} \circ g_4^{-1} \circ g_5 \circ g_4\circ g_3 \circ g_1^{-1}\circ g_2\circ g_1$. One can see that $g.P_i = P_i$ for any $i=1,\ldots,m$ and $g.P = Q$. 
\end{proof}

\begin{lemma}\label{Grisha'sLemma}
Let $G = \langle H_1, K_d, K_a \rangle$ for some integers $a, d \in \Z_{>0}$.
Then $K_{d+a-1} \sse \overline{G}$.
\end{lemma}
\begin{proof}
\par Let us denote by $T$ one-dimensional torus, corresponding to the standard basis vector $e_2\in N.$ Consider the group $L=\Ad_{K_a}(\Ad_{K_d}(H_1))$.
Since $H_1$ is a one-parameter subgroup in $\Aut(\A^2)$, we see that $L$ is also a $\Ga$-subgroup. Thus we have $L = \exp(\K\partial)$ where $\partial$ is some locally nilpotent derivation. %написать в прелиминарис что это так?
In addition, $L \sse G$, so from Proposition~\ref{principal-part} we have $L_T \sse \overline{G}$.
Using Baker-Campbell-Hausdorff formula~(\ref{BCH_form}) we obtain
    $$
    \partial_T=-d x^{2d-1}\frac{\partial}{\partial y}-(a+d)x^{a+d-1}\frac{\partial}{\partial y}-ax^{2a-1}\frac{\partial}{\partial y}.
    $$
    
Similarly, we have 
\begin{align*}
    \Ad_{K_d}(H_1)_T=\exp(\K x^{2d-1}\frac{\partial}{\partial y}) \sse\overline{G} \text{ and } \Ad_{K_a}(H_1)_T=\exp(\K x^{2a-1}\frac{\partial}{\partial y}) \sse\overline{G}. 
\end{align*}

Since all elements of the group ${\exp(\K x^{a+d-1}\frac{\partial}{\partial y})}$ belong to $\langle L, \Ad_{K_a}(H_1)$, $\Ad_{K_d}(H_1) \rangle \sse \overline{G}$, we have $\exp(\K x^{a+d-1}\frac{\partial}{\partial y})\sse\overline{G}$. 
\end{proof}
\begin{lemma}\label{firstPart}
If $d_1,\ldots,d_m \in \Z_{\geqslant 0}$ and $\GCD(d_1-1,\ldots, d_m-1)=1$, then there exists $n_0 \in \Z_{>0}$ such that $K_n \sse \overline{\la H_1,K_{d_1},\ldots,K_{d_m}\ra}$ for all $n > n_0$.
\end{lemma}
\begin{proof}
From Lemma~\ref{Grisha'sLemma} it follows that for any non-negative integers $l_1,\ldots,l_m$ we have 
$$K_{(d_1-1)l_1+\ldots +(d_m-1)l_m + 1} \sse \overline{\la H_1,K_{d_1},\ldots,K_{d_m}\ra}.
$$
Since $\GCD(d_1-1,\ldots, d_m-1)=1$ there exists $n_0 \in \Z_{>0}$ such that for all $ n > n_0$ we can find  non-negative integers $l_1,\ldots,l_m$ such that $(d_1-1)l_1+\ldots +(d_m-1)l_m = n-1$. This completes the proof.
\end{proof}
\begin{theorem}\label{H1}
    The group $\la H_1,K_{d_1},\ldots,K_{d_m}\ra$ acts on $\A^2\setminus\{0\}$ infinitely transitively if and only if $\GCD(d_1-1,\ldots, d_m-1)=1.$
\end{theorem}

\begin{proof}
Lemmata~\ref{firstPart} and~\ref{secondPart} constitute sufficiency; to prove necessity we follow \mbox{\cite[Section 5.1]{AKZ}:} let $m=\GCD(d_1-1,\ldots,d_m-1)>1$. Let $\omega$ be a primitive root of unity of degree $m$. Consider the set  
    $$
        S=\{(P,\omega P)\;\vert\; P\in\A^2\}.
    $$
    Since $\omega^{d_i} = \omega$ holds for all $i$, $S$ is a $G$-invariant subset in $\A^2\times\A^2.$ Thus $G$ does not act $2$-transitively on $\A^2\setminus \{0\}$.
\end{proof}

Denote $x^{\underline{k}}=x(x-1)\ldots(x-k+1)$.

\begin{lemma}\label{Aliska'sLemma}
Let $H_c = \exp(\K\partial_{r, \varepsilon})$ and $K_{d_i} = \exp(\K\partial_{\rho, e_i})$ for $i = 1,\ldots,k$. Consider the commutator
\begin{gather*}
{\partial_{f, s}=\left[\,\left[\dots\left[\,\left[\partial_{r,\varepsilon},\, \partial_{\rho, e_1}\right],\, \partial_{\rho, e_2}\right]\ldots\right],\, \partial_{\rho, e_k}\right].}\text{ Then}\\ s = \varepsilon + e_1 + e_2 + \ldots + e_k,\;\text{and } f =(-1)^{k+1}(\langle\rho, \varepsilon\rangle^{\underline{k-1}}\, \langle r,e_1+\dots+e_k\rangle\rho -\langle\rho, \varepsilon\rangle^{\underline{k}}\,r).\end{gather*}
\end{lemma}
\begin{proof}
    We prove this result by induction on \(k\). For \(k=1\) the result follows from the direct computation:
    \begin{gather*}
        \left[\partial_{r,\varepsilon},\, \partial_{\rho, e_1}\right](\chi^m)=\pd_{r,\ve}\circ \pd_{\rho,e_1}(\chi^m)-\pd_{\rho,e_1}\circ\pd_{r,\ve}(\chi^m)=\\
        =\pd_{r,\ve}(\la \rho, m\ra \chi^{m+e_1})-\pd_{\rho,e_1}(\la r, m\ra \chi^{m+\ve})=\\
        (\la\rho,m\ra\cdot\la r,m+e_1\ra-\la\rho,m+\ve\ra\cdot\la r,m\ra)\cdot\chi^{m+e_1+\ve}=\\
        =\la \la r,e_1\ra\rho-\la\rho,\ve\ra r, m\ra \chi^{m+e_1+\ve}.
    \end{gather*}
    
    Now assume that the result holds for \(k\le n\). We wish to prove it for \(k=n+1\). For the sake of brevity we denote by \(s_k,f_k\) respectively the degree and the linear part of the~\(k\)-iterated commutator. By definition we have
    \begin{gather}
        \begin{split}
            \pd_{f_{n+1},s_{n+1}}(\chi^m)=\left[\left[\,\left[\dots\left[\,\left[\partial_{r,\varepsilon},\, \partial_{\rho, e_1}\right],\, \partial_{\rho, e_2}\right]\ldots\right]\, \partial_{\rho, e_n}\right],\pd_{\rho,e_{n+1}}\right](\chi^m)=\\
        =[\pd_{f_n,s_n},\pd_{\rho,e_{n+1}}](\chi^m)=\la \la f_n,e_{n+1}\ra \rho-\la \rho, s_n\ra f_n,m\ra \chi^{m+e_{n+1}+s_n}.\label{form:n+1_der}
        \end{split}
    \end{gather}
    By the inductive hypothesis it is clear that \(s_{n+1}=\ve+e_1+\ldots+e_n+e_{n+1}\). To compute~\(f_{n+1}\) we have to expand the linear term of the derivation \(\pd_{f_{n+1},s_{n+1}}\) in formula \eqref{form:n+1_der}. This computation is straightforward:
    \begin{gather*}
       f_{n+1}=(-1)^{n+1}\Bigg\la \la \rho, \ve\ra^{\ul{n-1}}\la r, e_1+\ldots+e_n\ra \rho - \la \rho,\ve\ra^{\ul{n}}r,e_{n+1}\Bigg\ra \cdot \rho +\\
        +(-1)^{n+2}\Bigg\la \rho, \ve+e_1+\ldots+e_n\Bigg\ra \cdot \Bigg(\la \rho, \ve\ra^{\ul{n-1}}\la r, e_1+\ldots+e_n\ra \rho - \la \rho,\ve\ra^{\ul{n}}r\Bigg)=\\
        =(-1)^{n+1}\la \rho, \ve\ra^{\ul{n-1}}\la r, e_1+\ldots+e_n\ra\cdot\la\rho,e_{n+1}\ra\cdot \rho+(-1)^{n}\la \rho,\ve\ra^{\ul{n}}\cdot\la r,e_{n+1}\ra \rho+\\
        +(-1)^n\Bigg\la \rho, \ve+e_1+\ldots+e_n\Bigg\ra\cdot \la \rho, \ve\ra^{\ul{n-1}}\la r, e_1+\ldots+e_n\ra \rho +\\+ (-1)^{n+1}\Bigg\la \rho, \ve+e_1+\ldots+e_n\Bigg\ra\cdot \la \rho,\ve\ra^{\ul{n}}r
    \end{gather*}
    Note that \(\la \rho,e_k\ra =-1\) for all \(k,\) thus we have
    \begin{gather*}
       f_{n+1}=(-1)^{n+1}\la\rho,\ve\ra^{\ul{n+1}}r+(-1)^{n+1}\big(\la\rho,\ve\ra^{\ul{n}}\la r,e_{n+1}\ra\rho+\la\rho,\ve\ra^{\ul{n}}\la r,e_1+\ldots+e_n\ra\big)=\\
        =(-1)^{n+2}(\la\rho,\ve\ra^{\ul{n}}\la r,e_1+\ldots+e_n+e_{n+1}\ra\rho-\la\rho,\ve\ra^{\ul{n+1}}r).
    \end{gather*}
    We have proven the inductive step and hence the statement of the lemma.
\end{proof}
The following lemma is a specialization of Proposition \ref{principal-part} to the case of the affine plane~\(\A^2\) and subgroup generated by root subgroups.
\begin{lemma}\label{lemma:closures}
   Consider a subgroup \(G\ss\Aut(\A^2)\) normalized by one-dimensional torus with the linear part \(\rho=(0,1).\)
   Let \(L=\exp(\K\pd)\) be \(\Ga\)-subgroup of \(G.\) Here \({\pd\in\LND(\K[x_1,x_2]).}\) Let \(\pd_x\) be the principal part of \(\pd\) corresponding to a one-parametric subgroup with the linear part \(\rho=(0,1).\) Since~\(\pd_x\) is a summand of \(\pd\) there is a decomposition:
   \begin{equation}\label{eq:sum}
       \pd_x=\sum \pd_{\rho,e_i},\quad \text{where } e_i=(a_i,0).
   \end{equation}
   Here each \(\pd_{\rho,e_i}\) is homogeneous and belongs to the decomposition of \(\pd\) into the sum of homogeneous derivations.
   \par Then each of \(\pd_{\rho,e_i}\) is locally nilpotent and each of \(\Ga\) subgroups \(\exp(\K\pd_{\rho,e_i})\) is a subgroup of \(\ol{G}.\)
\end{lemma}
\begin{proof}
    Let \(e_{\max},e_{\min}\) be vectors \(e_i\) with minimal and maximal values of \(a_i\) in sum \eqref{eq:sum} respectively. 
    First note that Proposition \ref{principal-part} implies that \(\pd_x,\pd_{\rho,e_{\max}},\) and \(\pd_{\rho,e_{\min}}\) are locally nilpotent and subgroups \(\exp(\K\pd_{\rho}),\exp(\K\pd_{\rho,e_{\max}})\), \(\exp(\K\pd_{\rho,e_{\min}})\) are in \(\ol{G}.\) Now note that all \(\pd_{\rho,e_i}\) correspond to the same root \(\rho\) and thus commute. Thus the \(\Ga\)-subgroup \[H'=\exp(\K(\pd_x-\pd_{\rho,e_{\max}}-\pd_{\rho,e_{\min}}))\] is also  in \(\ol{G}.\) Now it remains to note that the length of sum \eqref{eq:sum} has decreased and hence the statement of lemma follows by induction applied to \(\ol{G}.\)
\end{proof}
\begin{proposition}\label{lemma:commutators}
    If the vector \(e=(a,-1)\) lies in the cone~\({\Z_{\ge 0}\la e_1,\ldots,e_k,\ve_1,\ldots,\ve_l\ra,}\)
    then the group \(K_e\) lies in the closure of the group
    \[
        G=\la K_{e_1},\ldots,K_{e_k},H_{\ve_1},\ldots,H_{\ve_l}\ra.
    \]
\end{proposition}
\begin{proof}
    By Lemma \ref{lemma:closures} it is sufficient to prove that the nilpotent derivation corresponding to the group \(K_e\) can be expressed as a homogeneous summand in a nilpotent derivation of the form:
    \begin{equation}\label{eq:ad_der}
        \widetilde{\pd}=\Ad_{G_{i_1}}(\Ad_{G_{i_2}}(\ldots \Ad_{G_{i_{n-1}}}(\pd_{i_{n}})\ldots)).
    \end{equation}
    Here \(G_{i_j}\) are \(\Ga\)-subgroups from the collection \(\{K_{e_s},H_{\ve_t}\}\) and the derivation \(\pd_{i_n}\) is 
    the LND corresponding to \(\Ga\)-subgroup \(G_{i_n}.\) 
    Below we construct derivation of form \eqref{eq:ad_der}. 
    \par Assume that the non-negative linear combination expressing \(e\) in terms of \(e_i,\ve_j\) is
    \begin{equation}\label{form:comb}
        e=\sum_{i=1}^N v_i+\sum_{j=1}^M \mu_j,\quad v_i\in\{e_s\}^k_{s=1},\; \mu_j\in\{\ve_t\}^l_{t=1}.
    \end{equation}
    Let us denote by \(\phi_1\) the sum:
    \[
        \phi_1=\sum_{i=1}^{\la \mu_1,\rho\ra+1} v_i+\mu_1.
    \]
    Let us also denote by \(n_1\) the number \(\la \mu_1,\rho\ra+1.\)
    Now we define \(f_i,n_i\) for \(i>1\) via formulas:
    \[
        \phi_i=\phi_{i-1}+\sum_{j=n_{i-1}+1}^{n_i} v_j + \mu_i,\quad n_i=n_{i-1}+\la\mu_i,\rho\ra + 2.
    \]
     Thus we have a sequence of subsums \(\{\phi_i\}_{i=1}^M\) in \eqref{form:comb} with \(\phi_M=e.\) 
    Conceptually \(\phi_i\) are constructed by picking one vector \(\mu_j\) to increase the second coordinate of \( \phi_{i-1}\) and then adding vectors \(v_i\) to gradually decrease the second coordinate and obtain the next vector~\(\phi_{i}.\) Note that all vectors  \(\phi_i\) have the form \((a_i,-1).\)
    \par For notational convenience we will denote \(\pd_{\rho,e}\) simply by \(\pd_e\) and similarly \(\pd_{r,\ve}\) by \(\pd_\ve.\) We also put \(\phi_0 = e_1\) and \(n_0=1\).  Now for all \(i\) from \(1\) to \(M\) we have identities
    \begin{gather}\label{eq:summand}
        \left[\left[\left[\ldots\left[\left[\left[\pd_{\mu_{i}},\pd_{\phi_{i-1}}\right],\pd_{v_{n_{i-1}+1}}\right],\pd_{v_{n_{i-1}+2}}\right]\ldots\right],\pd_{v_{n_i-1}}\right],\pd_{v_{n_i}}\right]=\pd_{\phi_{i}}.
    \end{gather}
    Let us further investigate the derivation \(\pd_{\phi_i}.\) By lemma \ref{Aliska'sLemma} it is a derivation of the form~\(\pd_{f,\phi_i},\) where \(f\) is a vector:
    \[
        f=(-1)^{\la \mu_i,\rho\ra + 1}(\langle\rho, \mu_i\rangle^{\underline{\la \mu_i,\rho\ra}}\, \langle r,\phi_i+\nu_{n_{i-1}+1}+\dots+\nu_{n_i}\rangle\rho -\langle\rho, \mu_i\rangle^{\underline{\la \mu_i,\rho\ra + 1}}\,r).
    \]
    Recall, that \(x^{\ul{i}}=x(x-1)\ldots(x-i+1)\), so \(f\) has much simpler form:
    \[f=(-1)^{\la \mu_i,\rho\ra + 1}\langle\rho, \mu_i\rangle!\, \langle r,\phi_i+\nu_{n_{i-1}+1}+\dots+\nu_{n_i}\rangle\rho.\]
    Hence \(f\neq 0\) and so \(\pd_{\phi_i}\) is non-zero.
    We now consider an LND
    \begin{equation}\label{eq:Ads}
        \Ad_{\exp\left(\K \pd_{v_{n_i}}\right)}\left(\ldots\left(\Ad_{\exp\left(\K \pd_{v_{n_{i-1}+1}}\right)}\left(\Ad_{\exp\left(\K \pd_{\phi_{i-1}}\right)}\left(\pd_{\mu_{i}}\right)\right)\right)\ldots\right).
    \end{equation}
    By Baker--Campbell--Hausdorff formula \eqref{eq:BCH_formula} it is clear that in \eqref{eq:Ads} there is a summand~\(\vk_i\pd_{\phi_i}.\) with some integer coefficient which comes from the number of different ways one obtains \(\phi_i\) as a sum of \(\nu_j\)'s. Note, however, that the formula for \(f\) from \ref{Aliska'sLemma} depends only on the number \(e_i's\) involved and sum of degrees, but not on specific order. The number of \(e_i's\) in turn depends only on a specific meet point of paths. Thus at each point two different paths meet they correspond to the same derivation. Hence,~\(\pd_{\phi_i}\)~is a summand in \eqref{eq:Ads} with non-zero coefficient.
    \par Clearly, \(\exp(\pd_{\phi_0})\) is a \(\Ga\)-subgroup in \(\ol{G}.\) Now by Lemma \ref{lemma:closures} if \(\phi_i\) defines a \(\Ga\)-subgroup in \(\ol{G}\) then \(\phi_{i+1}\) also does. Thus \(\phi_M=e\) defines a \(\Ga\)-subgroup in \(\ol{G}.\)
\end{proof}
\begin{theorem}\label{thm2}
    The group $G=\la H_{c_1},\ldots, H_{c_s},K_{d_1},\ldots,K_{d_m}\ra$ acts on $\A^2\setminus\{0\}$ infinitely transitively if and only if
    $$\Z\left\la \begin{pmatrix}-1 & c_1\end{pmatrix},\ldots,\begin{pmatrix}-1 & c_s\end{pmatrix},\begin{pmatrix}d_1& -1\end{pmatrix},\ldots,\begin{pmatrix}d_m & -1\end{pmatrix}\right\ra=M.$$
\end{theorem}
\begin{proof}
    Lemma \ref{Aliska'sLemma} and Proposition \ref{lemma:commutators} constitute sufficiency. To prove necessity we follow~\mbox{\cite[\S 5.1]{AKZ}}. Put~\(m=[M:\Z\la e_1,\ldots,e_k,\ve_1,\ldots,\ve_l\ra]\) and let \(\omega\) be a primitive root of unity of degree \(m.\) Consider a set
    \[
        S=\left\{\left(\left(z,w\right),\left(\omega z,\omega^{x(e_1)} w\right)\right)\;\bigg\vert\; (z,w) \in \A^2\right\}. 
    \]
    We now show that \(S\) is invariant. Denote by \(x(v)\) the first coordinate of the vector \(v\) and denote by \(y(v)\) the second coordinate. First we show that:
    \begin{enumerate}[(i)]
        \item \(\omega^{x(e_i)\cdot y(\ve_j)}=\omega\) for all \(i=1,\ldots, k,\; j=1,\ldots,l;\)\label{invset:one}
        \item \(\omega^{x(e_1)}=\omega^{x(e_i)}\) for all \(i=1,\ldots,k.\)\label{invset:two}
    \end{enumerate}
    To see \ref{invset:one} observe that \(x(e_i)\cdot y(\ve_j)-1\) equals the index of the subgroup \(\Z\la e_i,\ve_j\ra\) in \(\Z^2\) and thus is divisible by \(m.\) To prove \ref{invset:two} one notes that \(m\) is a common divisor of \(x(e_1)\cdot y(\ve_1)-1\) and \(x(e_i)\cdot y(\ve_1)-1\) by the previous observation. Thus \(m\) is a common divisor of \((x(e_1)-x(e_i))\cdot y(\ve_1)\) and \(x(e_i)\cdot y(\ve_1)-1.\) Since \(y(\ve_1)\) is coprime with \(x(e_i)\cdot y(\ve_1)-1\) it follows that \(m\) divides \(x(e_1)-x(e_i).\)
    \par Now note that \ref{invset:one} and \ref{invset:two} imply invariance of the set \(S\). Indeed, consider an action of \(H_{\ve_j}\) and \(K_{e_i}\) on \(S.\) We have:
    \begin{gather*}
          H_{\ve_j}(\alpha).S=\left\{\left(\left(z+\alpha w^{y(\ve_j)},w\right),\left(\omega z+\alpha\omega^{x(e_1) y(\ve_j)}w^{y(\ve_j)},\omega^{x(e_1)} w\right)\right)\;\bigg\vert\; (z,w) \in \A^2\right\}=\\=\left\{\left(\left(z+\alpha w^{y(\ve_j)},w\right),\left(\omega z+\alpha\omega w^{y(\ve_j)},\omega^{x(e_1)} w\right)\right)\;\bigg\vert\; (z,w) \in \A^2\right\}=S;\\
          K_{e_i}(\beta).S=\left\{\left(\left(z,w+\beta z^{x(e_i)}\right),\left(\omega z,\omega^{x(e_1)} w+\beta \omega^{x(e_i)}z^{x(e_i)}\right)\right)\;\bigg\vert\; (z,w) \in \A^2\right\}=\\
          =\left\{\left(\left(z,w+\beta w^{x(e_i)}\right),\left(\omega z,\omega^{x(e_1)} w+\beta \omega^{x(e_1)}z^{x(e_i)}\right)\right)\;\bigg\vert\; (z,w) \in \A^2\right\}=S.
    \end{gather*}
    Thus the set \(S\) is indeed invariant and the group \(G\) does not act even \(2\)-transitively on~$\A^2\setminus \{0\}.$
\end{proof}
\section{Examples and Questions}\label{sec:questions}

Theorem~\ref{H1} provides plethora of examples of ind-groups acting infinitely transitively on $\A^2\setminus\{0\}$.
Note that although $\overline{\la H_1,K_{d_1},\ldots,K_{d_m}\ra}$ and $\overline{\la H_1,K_{c_1},\ldots,K_{c_n}\ra}$ can coincide for
different sets of numbers $(d_1,\ldots,d_m)$ and $(c_1,\ldots,c_n)$, they do not always coincide. For example, if $c=\min\{c_1,\ldots,c_n\} \neq \min\{d_1,\ldots,d_m\} = d$, we have $$G_d = \overline{\la H_1,K_{d_1},\ldots,K_{d_m}\ra} \neq \overline{\la H_1,K_{c_1},\ldots,K_{c_n}\ra} = G_c.$$
Indeed, if $c > d$ we have $K_d \sse G_d$ and $K_d \not\sse G_c$, in the case with $c < d$ we have $K_c \sse G_c$ and $K_c \not\sse G_d$.
Thus we obtain a countable family of different ind-groups acting infinitely transitively on~$\A^2\setminus\{0\}$. 
\begin{example}
There is only one subgroup generated by $H_1$ and one other root group that acts infinitely transitively on $\A^2\setminus\{0\}$, it is the group $\langle H_1, K_2\rangle$. However, there is a lot of subgroups generated by $H_1$ and two other root groups that act infinitely transitively on its open orbit $\A^2\setminus\{0\}$. Indeed, for any positive integer $k \ge 3$ the group $\langle H_1, K_{k}, K_{k+1}\rangle$ acts on $\A^2\setminus\{0\}$ infinitely transitively, since $\gcd(k-1, k) = 1$. Notice that for $k_1 \neq k_2$ we have
$$\overline{\langle H_1, K_{k_1}, K_{k_1+1}\rangle} \neq \overline{\langle H_1, K_{k_2}, K_{k_2+1}\rangle}. $$
\end{example}
\begin{example}
The proof of Theorem~\ref{H1} also gives us a simple method to prove that some groups do not act $2$-transitively on $\A^2\setminus\{0\}$. For example, for any number of odd numbers $d_1,\ldots,d_m$ the group $\la H_1,K_{d_1},\ldots,K_{d_m} \ra$ does not act $2$-transitively on $\A^2\setminus\{0\}$. Indeed, since $2$ divides $\gcd(d_1-1,\ldots,d_m-1)$, we have $\gcd(d_1-1,\ldots,d_m-1) > 1$, and now we can use the proof of Theorem~\ref{H1}.
\end{example}
\begin{example} Here we provide an in-detail application of the algorithm described in the proof of proposition \ref{lemma:commutators}. Assume that we are given a collection of four \(\Ga\) groups \(K_2, K_3, K_4, H_2.\) We want to show that the group \(K_12\) belongs to \(\ol{\la K_2, K_3, K_4, H_2\ra}.\) As was shown in the course of proving theorem \(2\) it is necessary and sufficient to show that the vector \((12,-1)\) defines an LND, which is a summand in an LND of the form \(\Ad_{G_{i_1}}(\Ad_{G_{i_2}}(\ldots \Ad_{G_{i_{n-1}}}(\pd_{i_{n}})\ldots)).\) Here \(G_i's\) are groups from the collection \(\{K_2,K_3,K_4,H_2\}\) and \(\pd_{i_n}\) is one of the derivations corresponding to these groups. Clearly, the vector \((12,-1)\) can be expressed as a positive linear combination of the vectors corresponding to \(K\)'s and \(H_2.\) Such a combination is given by the sum:
\begin{gather*}
    (12,-1)=\nu_1+\nu_2+\nu_3+\nu_4+\mu_1,\\ \mu_1=(-1,3), \nu_1=(2,-1),\;\nu_2=(3,-1),\;\nu_3=(4,-1), \nu_4=(4,-1).
\end{gather*}

According to the algorithm of proposition \ref{lemma:commutators} we should consider the sum:
\[
    \phi_1=\sum_{i=1}^{\la (-1,3),(0,1)\ra+1}\nu_i+\mu_1=(12,-1).
\]
Now we compute action of \(\Ad\)-operators on \(\pd_{\nu_1}.\) This can be seen in the figure \ref{fig:comm}. We get an LND \[
D=\Ad_{K_4}(\Ad_{K_4}(\Ad_{K_3}(\Ad_{K_2}(\pd_{H_2})))).
\]
By construction \(\exp(\K D)\sse \ol{\la K_2, K_3, K_4, H_2\ra}.\)
It is clear, that \(D\) satisfies conditions of lemma \ref{lemma:closures}. Thus it \(\pd_{K_{12}}\) defines a \(\Ga\)-subgroup in \(\ol{\la K_2, K_3, K_4, H_2\ra}\) and we are done.
\begin{figure}[H]
    \centering
    \input{diagram_app}
    \caption{The whole figure represents an entire LND \[\Ad_{K_4}(\Ad_{K_3}(\Ad_{K_2}(\Ad_{K_1}(\pd_{H_2})))).\] The solid picewise linear path represents a commutator expression of the form \eqref{eq:ad_der} which is equal to derivation \(\pd_{(12,-1)}.\)}
    \label{fig:comm}
\end{figure}
\end{example}

Now we formulate a generalization of our results, which will be the topic of our subsequent work.
Let $X$ be an affine toric variety. Recall that its coordinate ring admits a decomposition
$$
\K[X] = \bigoplus_{m\in M\cap\, \sigma^\vee}\K\chi^m,
$$
where $\sigma^\vee$ is a cone in $M$, corresponding to $X$, see \cite{Co}, \cite{Fu}. %исправить ссылки в этом абзаце
If $\partial_{e}$ is a locally nilpotent derivation of $\K[X]$, then $\Ad(\K\partial_e)$ is a $\Ga$-subgroup of $\Aut(X)$, see \cite[\S 2]{LiC1}.
First we aim to generalize Theorem~\ref{thm2} for an affine toric surface.
\begin{conjecture}\label{toricsurface}
Let $X$ be an affine toric surface. Let \(\sigma\) be a cone corresponding to~\(X.\) Then the group~${G =\! \langle\exp(\K\partial_{e_1}),\ldots,\exp(\K\partial_{e_k})\rangle}$~acts on its open orbit infinitely transitively if and only if \[\Z\langle e_1,\ldots,e_k\rangle = \Z\la\sigma\ra.\]
\end{conjecture}
Here the problem can be solved directly by resolving a singular toric surface with a series of equivariant blow-ups thus obtaining a gluing of several copies of \(\A^2.\) A sharper statement generalizing Conjecture \ref{toricsurface} is the following:
\begin{conjecture}
    Let $X$ be an affine toric variety. Let \(\sigma\) be a cone corresponding to~\(X.\) The group~${G =\! \langle\exp(\K\partial_{e_1}),\ldots,\exp(\K\partial_{e_k})\rangle}$~act on its open orbit infinitely transitively if and only if \[\Z\langle e_1,\ldots,e_k\rangle = \Z\la\sigma\ra.\]
\end{conjecture}
However here, the proof should rely on a general formalism that we plan to develop in our future papers. We end this section with the following 
\begin{question}
	Let \(X\sse Y\) be a toric subvariety of a toric variety \(Y.\) Assume that \(G\) is a group of automorphisms of \(X\) generated by root subgroups. When is there an automorphism group \(\wt{G}\) of \(Y\) generated by \(G\) and some other root subgroups such that \(\wt{G}\) acts infinitely transitively on \(Y\)?
\end{question}

\end{document}

%% file: diagram_app.tex
\tikzset{every picture/.style={line width=0.75pt}} %set default line width to 0.75pt        

\begin{tikzpicture}[x=0.75pt,y=0.75pt,yscale=-1,xscale=1]
%uncomment if require: \path (0,302); %set diagram left start at 0, and has height of 302

%Shape: Grid [id:dp4135190192782685] 
\draw  [draw opacity=0] (1.67,1) -- (641.67,1) -- (641.67,302) -- (1.67,302) -- cycle ; \draw  [color={rgb, 255:red, 234; green, 226; blue, 226 }  ,draw opacity=1 ] (1.67,1) -- (1.67,302)(21.67,1) -- (21.67,302)(41.67,1) -- (41.67,302)(61.67,1) -- (61.67,302)(81.67,1) -- (81.67,302)(101.67,1) -- (101.67,302)(121.67,1) -- (121.67,302)(141.67,1) -- (141.67,302)(161.67,1) -- (161.67,302)(181.67,1) -- (181.67,302)(201.67,1) -- (201.67,302)(221.67,1) -- (221.67,302)(241.67,1) -- (241.67,302)(261.67,1) -- (261.67,302)(281.67,1) -- (281.67,302)(301.67,1) -- (301.67,302)(321.67,1) -- (321.67,302)(341.67,1) -- (341.67,302)(361.67,1) -- (361.67,302)(381.67,1) -- (381.67,302)(401.67,1) -- (401.67,302)(421.67,1) -- (421.67,302)(441.67,1) -- (441.67,302)(461.67,1) -- (461.67,302)(481.67,1) -- (481.67,302)(501.67,1) -- (501.67,302)(521.67,1) -- (521.67,302)(541.67,1) -- (541.67,302)(561.67,1) -- (561.67,302)(581.67,1) -- (581.67,302)(601.67,1) -- (601.67,302)(621.67,1) -- (621.67,302) ; \draw  [color={rgb, 255:red, 234; green, 226; blue, 226 }  ,draw opacity=1 ] (1.67,1) -- (641.67,1)(1.67,21) -- (641.67,21)(1.67,41) -- (641.67,41)(1.67,61) -- (641.67,61)(1.67,81) -- (641.67,81)(1.67,101) -- (641.67,101)(1.67,121) -- (641.67,121)(1.67,141) -- (641.67,141)(1.67,161) -- (641.67,161)(1.67,181) -- (641.67,181)(1.67,201) -- (641.67,201)(1.67,221) -- (641.67,221)(1.67,241) -- (641.67,241)(1.67,261) -- (641.67,261)(1.67,281) -- (641.67,281)(1.67,301) -- (641.67,301) ; \draw  [color={rgb, 255:red, 234; green, 226; blue, 226 }  ,draw opacity=1 ]  ;
%Straight Lines [id:da9602731337684484] 
\draw    (81,241) -- (539,241) ;
\draw [shift={(541,241)}, rotate = 180] [color={rgb, 255:red, 0; green, 0; blue, 0 }  ][line width=0.75]    (21.86,-6.58) .. controls (13.9,-2.79) and (6.61,-0.6) .. (0,0) .. controls (6.61,0.6) and (13.9,2.79) .. (21.86,6.58)   ;
%Straight Lines [id:da5056814545748868] 
\draw    (81,241) -- (81,83) ;
\draw [shift={(81,81)}, rotate = 90] [color={rgb, 255:red, 0; green, 0; blue, 0 }  ][line width=0.75]    (21.86,-6.58) .. controls (13.9,-2.79) and (6.61,-0.6) .. (0,0) .. controls (6.61,0.6) and (13.9,2.79) .. (21.86,6.58)   ;
%Straight Lines [id:da6611664236587004] 
\draw [line width=2.25]    (81,241) -- (62.26,184.79) ;
\draw [shift={(61,181)}, rotate = 71.57] [color={rgb, 255:red, 0; green, 0; blue, 0 }  ][line width=2.25]    (17.49,-5.26) .. controls (11.12,-2.23) and (5.29,-0.48) .. (0,0) .. controls (5.29,0.48) and (11.12,2.23) .. (17.49,5.26)   ;
%Straight Lines [id:da6571611826718688] 
\draw [color={rgb, 255:red, 245; green, 166; blue, 35 }  ,draw opacity=1 ][line width=2.25]    (61,181) -- (97.42,199.21) ;
\draw [shift={(101,201)}, rotate = 206.57] [color={rgb, 255:red, 245; green, 166; blue, 35 }  ,draw opacity=1 ][line width=2.25]    (17.49,-5.26) .. controls (11.12,-2.23) and (5.29,-0.48) .. (0,0) .. controls (5.29,0.48) and (11.12,2.23) .. (17.49,5.26)   ;
%Straight Lines [id:da9760920273127077] 
\draw [color={rgb, 255:red, 121; green, 152; blue, 89 }  ,draw opacity=1 ][line width=2.25]    (101,201) -- (177.12,220.03) ;
\draw [shift={(181,221)}, rotate = 194.04] [color={rgb, 255:red, 121; green, 152; blue, 89 }  ,draw opacity=1 ][line width=2.25]    (17.49,-5.26) .. controls (11.12,-2.23) and (5.29,-0.48) .. (0,0) .. controls (5.29,0.48) and (11.12,2.23) .. (17.49,5.26)   ;
%Straight Lines [id:da7408898998621144] 
\draw [color={rgb, 255:red, 139; green, 87; blue, 42 }  ,draw opacity=1 ][line width=2.25]    (181,221) -- (317.04,240.43) ;
\draw [shift={(321,241)}, rotate = 188.13] [color={rgb, 255:red, 139; green, 87; blue, 42 }  ,draw opacity=1 ][line width=2.25]    (17.49,-5.26) .. controls (11.12,-2.23) and (5.29,-0.48) .. (0,0) .. controls (5.29,0.48) and (11.12,2.23) .. (17.49,5.26)   ;
%Straight Lines [id:da37069125939583736] 
\draw [color={rgb, 255:red, 56; green, 163; blue, 139 }  ,draw opacity=1 ][line width=2.25]    (341,141) -- (477.04,160.43) ;
\draw [shift={(481,161)}, rotate = 188.13] [color={rgb, 255:red, 56; green, 163; blue, 139 }  ,draw opacity=1 ][line width=2.25]    (17.49,-5.26) .. controls (11.12,-2.23) and (5.29,-0.48) .. (0,0) .. controls (5.29,0.48) and (11.12,2.23) .. (17.49,5.26)   ;
%Straight Lines [id:da6764180724027659] 
\draw [color={rgb, 255:red, 139; green, 87; blue, 42 }  ,draw opacity=1 ] [dash pattern={on 0.84pt off 2.51pt}]  (101,201) -- (239.02,220.72) ;
\draw [shift={(241,221)}, rotate = 188.13] [color={rgb, 255:red, 139; green, 87; blue, 42 }  ,draw opacity=1 ][line width=0.75]    (10.93,-3.29) .. controls (6.95,-1.4) and (3.31,-0.3) .. (0,0) .. controls (3.31,0.3) and (6.95,1.4) .. (10.93,3.29)   ;
%Straight Lines [id:da46890260655112925] 
\draw [color={rgb, 255:red, 139; green, 87; blue, 42 }  ,draw opacity=1 ] [dash pattern={on 0.84pt off 2.51pt}]  (241,221) -- (379.02,240.72) ;
\draw [shift={(381,241)}, rotate = 188.13] [color={rgb, 255:red, 139; green, 87; blue, 42 }  ,draw opacity=1 ][line width=0.75]    (10.93,-3.29) .. controls (6.95,-1.4) and (3.31,-0.3) .. (0,0) .. controls (3.31,0.3) and (6.95,1.4) .. (10.93,3.29)   ;
%Straight Lines [id:da734322640882057] 
\draw [color={rgb, 255:red, 139; green, 87; blue, 42 }  ,draw opacity=1 ] [dash pattern={on 0.84pt off 2.51pt}]  (381,241) -- (519.02,260.72) ;
\draw [shift={(521,261)}, rotate = 188.13] [color={rgb, 255:red, 139; green, 87; blue, 42 }  ,draw opacity=1 ][line width=0.75]    (10.93,-3.29) .. controls (6.95,-1.4) and (3.31,-0.3) .. (0,0) .. controls (3.31,0.3) and (6.95,1.4) .. (10.93,3.29)   ;
%Straight Lines [id:da6173507949752711] 
\draw [color={rgb, 255:red, 245; green, 166; blue, 35 }  ,draw opacity=1 ] [dash pattern={on 0.84pt off 2.51pt}]  (101,201) -- (139.21,220.11) ;
\draw [shift={(141,221)}, rotate = 206.57] [color={rgb, 255:red, 245; green, 166; blue, 35 }  ,draw opacity=1 ][line width=0.75]    (10.93,-3.29) .. controls (6.95,-1.4) and (3.31,-0.3) .. (0,0) .. controls (3.31,0.3) and (6.95,1.4) .. (10.93,3.29)   ;
%Straight Lines [id:da08478646428316261] 
\draw [color={rgb, 255:red, 245; green, 166; blue, 35 }  ,draw opacity=1 ] [dash pattern={on 0.84pt off 2.51pt}]  (141,221) -- (179.21,240.11) ;
\draw [shift={(181,241)}, rotate = 206.57] [color={rgb, 255:red, 245; green, 166; blue, 35 }  ,draw opacity=1 ][line width=0.75]    (10.93,-3.29) .. controls (6.95,-1.4) and (3.31,-0.3) .. (0,0) .. controls (3.31,0.3) and (6.95,1.4) .. (10.93,3.29)   ;
%Straight Lines [id:da90727416661315] 
\draw [color={rgb, 255:red, 245; green, 166; blue, 35 }  ,draw opacity=1 ] [dash pattern={on 0.84pt off 2.51pt}]  (181,241) -- (219.21,260.11) ;
\draw [shift={(221,261)}, rotate = 206.57] [color={rgb, 255:red, 245; green, 166; blue, 35 }  ,draw opacity=1 ][line width=0.75]    (10.93,-3.29) .. controls (6.95,-1.4) and (3.31,-0.3) .. (0,0) .. controls (3.31,0.3) and (6.95,1.4) .. (10.93,3.29)   ;
%Straight Lines [id:da22379544282494868] 
\draw [color={rgb, 255:red, 121; green, 152; blue, 89 }  ,draw opacity=1 ] [dash pattern={on 0.84pt off 2.51pt}]  (141,221) -- (219.06,240.51) ;
\draw [shift={(221,241)}, rotate = 194.04] [color={rgb, 255:red, 121; green, 152; blue, 89 }  ,draw opacity=1 ][line width=0.75]    (10.93,-3.29) .. controls (6.95,-1.4) and (3.31,-0.3) .. (0,0) .. controls (3.31,0.3) and (6.95,1.4) .. (10.93,3.29)   ;
%Straight Lines [id:da9115378488932135] 
\draw [color={rgb, 255:red, 121; green, 152; blue, 89 }  ,draw opacity=1 ] [dash pattern={on 0.84pt off 2.51pt}]  (221,241) -- (299.06,260.51) ;
\draw [shift={(301,261)}, rotate = 194.04] [color={rgb, 255:red, 121; green, 152; blue, 89 }  ,draw opacity=1 ][line width=0.75]    (10.93,-3.29) .. controls (6.95,-1.4) and (3.31,-0.3) .. (0,0) .. controls (3.31,0.3) and (6.95,1.4) .. (10.93,3.29)   ;
%Straight Lines [id:da43445537256431654] 
\draw [color={rgb, 255:red, 139; green, 87; blue, 42 }  ,draw opacity=1 ] [dash pattern={on 0.84pt off 2.51pt}]  (141,221) -- (279.02,240.72) ;
\draw [shift={(281,241)}, rotate = 188.13] [color={rgb, 255:red, 139; green, 87; blue, 42 }  ,draw opacity=1 ][line width=0.75]    (10.93,-3.29) .. controls (6.95,-1.4) and (3.31,-0.3) .. (0,0) .. controls (3.31,0.3) and (6.95,1.4) .. (10.93,3.29)   ;
%Straight Lines [id:da7222156398142419] 
\draw [color={rgb, 255:red, 121; green, 152; blue, 89 }  ,draw opacity=1 ] [dash pattern={on 0.84pt off 2.51pt}]  (281,241) -- (419.02,260.72) ;
\draw [shift={(421,261)}, rotate = 188.13] [color={rgb, 255:red, 121; green, 152; blue, 89 }  ,draw opacity=1 ][line width=0.75]    (10.93,-3.29) .. controls (6.95,-1.4) and (3.31,-0.3) .. (0,0) .. controls (3.31,0.3) and (6.95,1.4) .. (10.93,3.29)   ;
%Straight Lines [id:da3468214027430466] 
\draw [color={rgb, 255:red, 56; green, 163; blue, 139 }  ,draw opacity=1 ] [dash pattern={on 0.84pt off 2.51pt}]  (61,181) -- (199.02,200.72) ;
\draw [shift={(201,201)}, rotate = 188.13] [color={rgb, 255:red, 56; green, 163; blue, 139 }  ,draw opacity=1 ][line width=0.75]    (10.93,-3.29) .. controls (6.95,-1.4) and (3.31,-0.3) .. (0,0) .. controls (3.31,0.3) and (6.95,1.4) .. (10.93,3.29)   ;
%Straight Lines [id:da5275644982325554] 
\draw [color={rgb, 255:red, 139; green, 87; blue, 42 }  ,draw opacity=1 ] [dash pattern={on 0.84pt off 2.51pt}]  (201,201) -- (339.02,220.72) ;
\draw [shift={(341,221)}, rotate = 188.13] [color={rgb, 255:red, 139; green, 87; blue, 42 }  ,draw opacity=1 ][line width=0.75]    (10.93,-3.29) .. controls (6.95,-1.4) and (3.31,-0.3) .. (0,0) .. controls (3.31,0.3) and (6.95,1.4) .. (10.93,3.29)   ;
%Straight Lines [id:da3947402742633964] 
\draw [color={rgb, 255:red, 121; green, 152; blue, 89 }  ,draw opacity=1 ] [dash pattern={on 0.84pt off 2.51pt}]  (181,221) -- (259.06,240.51) ;
\draw [shift={(261,241)}, rotate = 194.04] [color={rgb, 255:red, 121; green, 152; blue, 89 }  ,draw opacity=1 ][line width=0.75]    (10.93,-3.29) .. controls (6.95,-1.4) and (3.31,-0.3) .. (0,0) .. controls (3.31,0.3) and (6.95,1.4) .. (10.93,3.29)   ;
%Straight Lines [id:da07113331098089481] 
\draw [color={rgb, 255:red, 121; green, 152; blue, 89 }  ,draw opacity=1 ] [dash pattern={on 0.84pt off 2.51pt}]  (261,241) -- (339.06,260.51) ;
\draw [shift={(341,261)}, rotate = 194.04] [color={rgb, 255:red, 121; green, 152; blue, 89 }  ,draw opacity=1 ][line width=0.75]    (10.93,-3.29) .. controls (6.95,-1.4) and (3.31,-0.3) .. (0,0) .. controls (3.31,0.3) and (6.95,1.4) .. (10.93,3.29)   ;
%Straight Lines [id:da1320886233601236] 
\draw [color={rgb, 255:red, 139; green, 87; blue, 42 }  ,draw opacity=1 ] [dash pattern={on 0.84pt off 2.51pt}]  (341,221) -- (479.02,240.72) ;
\draw [shift={(481,241)}, rotate = 188.13] [color={rgb, 255:red, 139; green, 87; blue, 42 }  ,draw opacity=1 ][line width=0.75]    (10.93,-3.29) .. controls (6.95,-1.4) and (3.31,-0.3) .. (0,0) .. controls (3.31,0.3) and (6.95,1.4) .. (10.93,3.29)   ;
%Straight Lines [id:da34954461849128404] 
\draw [color={rgb, 255:red, 139; green, 87; blue, 42 }  ,draw opacity=1 ] [dash pattern={on 0.84pt off 2.51pt}]  (481,241) -- (619.02,260.72) ;
\draw [shift={(621,261)}, rotate = 188.13] [color={rgb, 255:red, 139; green, 87; blue, 42 }  ,draw opacity=1 ][line width=0.75]    (10.93,-3.29) .. controls (6.95,-1.4) and (3.31,-0.3) .. (0,0) .. controls (3.31,0.3) and (6.95,1.4) .. (10.93,3.29)   ;
%Shape: Circle [id:dp2777493295795497] 
\draw  [fill={rgb, 255:red, 208; green, 2; blue, 27 }  ,fill opacity=1 ] (56,161) .. controls (56,158.24) and (58.24,156) .. (61,156) .. controls (63.76,156) and (66,158.24) .. (66,161) .. controls (66,163.76) and (63.76,166) .. (61,166) .. controls (58.24,166) and (56,163.76) .. (56,161) -- cycle ;
%Shape: Circle [id:dp2148951975704757] 
\draw  [fill={rgb, 255:red, 208; green, 2; blue, 27 }  ,fill opacity=1 ] (56,141) .. controls (56,138.24) and (58.24,136) .. (61,136) .. controls (63.76,136) and (66,138.24) .. (66,141) .. controls (66,143.76) and (63.76,146) .. (61,146) .. controls (58.24,146) and (56,143.76) .. (56,141) -- cycle ;
%Shape: Circle [id:dp18187743052692562] 
\draw  [fill={rgb, 255:red, 208; green, 2; blue, 27 }  ,fill opacity=1 ] (56,121) .. controls (56,118.24) and (58.24,116) .. (61,116) .. controls (63.76,116) and (66,118.24) .. (66,121) .. controls (66,123.76) and (63.76,126) .. (61,126) .. controls (58.24,126) and (56,123.76) .. (56,121) -- cycle ;
%Shape: Circle [id:dp011444542763891508] 
\draw  [fill={rgb, 255:red, 208; green, 2; blue, 27 }  ,fill opacity=1 ] (56,101) .. controls (56,98.24) and (58.24,96) .. (61,96) .. controls (63.76,96) and (66,98.24) .. (66,101) .. controls (66,103.76) and (63.76,106) .. (61,106) .. controls (58.24,106) and (56,103.76) .. (56,101) -- cycle ;
%Shape: Circle [id:dp3611482189662334] 
\draw  [fill={rgb, 255:red, 208; green, 2; blue, 27 }  ,fill opacity=1 ] (56,81) .. controls (56,78.24) and (58.24,76) .. (61,76) .. controls (63.76,76) and (66,78.24) .. (66,81) .. controls (66,83.76) and (63.76,86) .. (61,86) .. controls (58.24,86) and (56,83.76) .. (56,81) -- cycle ;
%Shape: Circle [id:dp6242942767851934] 
\draw  [fill={rgb, 255:red, 208; green, 2; blue, 27 }  ,fill opacity=1 ] (56,221) .. controls (56,218.24) and (58.24,216) .. (61,216) .. controls (63.76,216) and (66,218.24) .. (66,221) .. controls (66,223.76) and (63.76,226) .. (61,226) .. controls (58.24,226) and (56,223.76) .. (56,221) -- cycle ;
%Shape: Circle [id:dp9246703788930677] 
\draw  [fill={rgb, 255:red, 208; green, 2; blue, 27 }  ,fill opacity=1 ] (56,201) .. controls (56,198.24) and (58.24,196) .. (61,196) .. controls (63.76,196) and (66,198.24) .. (66,201) .. controls (66,203.76) and (63.76,206) .. (61,206) .. controls (58.24,206) and (56,203.76) .. (56,201) -- cycle ;
%Shape: Circle [id:dp45141802700966815] 
\draw  [fill={rgb, 255:red, 208; green, 2; blue, 27 }  ,fill opacity=1 ] (56,181) .. controls (56,178.24) and (58.24,176) .. (61,176) .. controls (63.76,176) and (66,178.24) .. (66,181) .. controls (66,183.76) and (63.76,186) .. (61,186) .. controls (58.24,186) and (56,183.76) .. (56,181) -- cycle ;
%Shape: Circle [id:dp2510288027152042] 
\draw  [fill={rgb, 255:red, 208; green, 132; blue, 223 }  ,fill opacity=1 ] (96,261) .. controls (96,258.24) and (98.24,256) .. (101,256) .. controls (103.76,256) and (106,258.24) .. (106,261) .. controls (106,263.76) and (103.76,266) .. (101,266) .. controls (98.24,266) and (96,263.76) .. (96,261) -- cycle ;
%Shape: Circle [id:dp9289779183518898] 
\draw  [fill={rgb, 255:red, 208; green, 132; blue, 223 }  ,fill opacity=1 ] (116,261) .. controls (116,258.24) and (118.24,256) .. (121,256) .. controls (123.76,256) and (126,258.24) .. (126,261) .. controls (126,263.76) and (123.76,266) .. (121,266) .. controls (118.24,266) and (116,263.76) .. (116,261) -- cycle ;
%Shape: Circle [id:dp8879136836350473] 
\draw  [fill={rgb, 255:red, 208; green, 132; blue, 223 }  ,fill opacity=1 ] (136,261) .. controls (136,258.24) and (138.24,256) .. (141,256) .. controls (143.76,256) and (146,258.24) .. (146,261) .. controls (146,263.76) and (143.76,266) .. (141,266) .. controls (138.24,266) and (136,263.76) .. (136,261) -- cycle ;
%Shape: Circle [id:dp6283886192709063] 
\draw  [fill={rgb, 255:red, 208; green, 132; blue, 223 }  ,fill opacity=1 ] (156,261) .. controls (156,258.24) and (158.24,256) .. (161,256) .. controls (163.76,256) and (166,258.24) .. (166,261) .. controls (166,263.76) and (163.76,266) .. (161,266) .. controls (158.24,266) and (156,263.76) .. (156,261) -- cycle ;
%Shape: Circle [id:dp1580126518630589] 
\draw  [fill={rgb, 255:red, 208; green, 132; blue, 223 }  ,fill opacity=1 ] (176,261) .. controls (176,258.24) and (178.24,256) .. (181,256) .. controls (183.76,256) and (186,258.24) .. (186,261) .. controls (186,263.76) and (183.76,266) .. (181,266) .. controls (178.24,266) and (176,263.76) .. (176,261) -- cycle ;
%Shape: Circle [id:dp7351663207893804] 
\draw  [fill={rgb, 255:red, 208; green, 132; blue, 223 }  ,fill opacity=1 ] (196,261) .. controls (196,258.24) and (198.24,256) .. (201,256) .. controls (203.76,256) and (206,258.24) .. (206,261) .. controls (206,263.76) and (203.76,266) .. (201,266) .. controls (198.24,266) and (196,263.76) .. (196,261) -- cycle ;
%Shape: Circle [id:dp02633199630716232] 
\draw  [fill={rgb, 255:red, 208; green, 132; blue, 223 }  ,fill opacity=1 ] (216,261) .. controls (216,258.24) and (218.24,256) .. (221,256) .. controls (223.76,256) and (226,258.24) .. (226,261) .. controls (226,263.76) and (223.76,266) .. (221,266) .. controls (218.24,266) and (216,263.76) .. (216,261) -- cycle ;
%Shape: Circle [id:dp23963065272262662] 
\draw  [fill={rgb, 255:red, 208; green, 132; blue, 223 }  ,fill opacity=1 ] (236,261) .. controls (236,258.24) and (238.24,256) .. (241,256) .. controls (243.76,256) and (246,258.24) .. (246,261) .. controls (246,263.76) and (243.76,266) .. (241,266) .. controls (238.24,266) and (236,263.76) .. (236,261) -- cycle ;
%Shape: Circle [id:dp023069895898842674] 
\draw  [fill={rgb, 255:red, 208; green, 132; blue, 223 }  ,fill opacity=1 ] (256,261) .. controls (256,258.24) and (258.24,256) .. (261,256) .. controls (263.76,256) and (266,258.24) .. (266,261) .. controls (266,263.76) and (263.76,266) .. (261,266) .. controls (258.24,266) and (256,263.76) .. (256,261) -- cycle ;
%Shape: Circle [id:dp20742137147573558] 
\draw  [fill={rgb, 255:red, 208; green, 132; blue, 223 }  ,fill opacity=1 ] (276,261) .. controls (276,258.24) and (278.24,256) .. (281,256) .. controls (283.76,256) and (286,258.24) .. (286,261) .. controls (286,263.76) and (283.76,266) .. (281,266) .. controls (278.24,266) and (276,263.76) .. (276,261) -- cycle ;
%Shape: Circle [id:dp5034908017861729] 
\draw  [fill={rgb, 255:red, 208; green, 132; blue, 223 }  ,fill opacity=1 ] (296,261) .. controls (296,258.24) and (298.24,256) .. (301,256) .. controls (303.76,256) and (306,258.24) .. (306,261) .. controls (306,263.76) and (303.76,266) .. (301,266) .. controls (298.24,266) and (296,263.76) .. (296,261) -- cycle ;
%Shape: Circle [id:dp728894251459223] 
\draw  [fill={rgb, 255:red, 208; green, 132; blue, 223 }  ,fill opacity=1 ] (316,261) .. controls (316,258.24) and (318.24,256) .. (321,256) .. controls (323.76,256) and (326,258.24) .. (326,261) .. controls (326,263.76) and (323.76,266) .. (321,266) .. controls (318.24,266) and (316,263.76) .. (316,261) -- cycle ;
%Shape: Circle [id:dp191106798940887] 
\draw  [fill={rgb, 255:red, 208; green, 132; blue, 223 }  ,fill opacity=1 ] (336,261) .. controls (336,258.24) and (338.24,256) .. (341,256) .. controls (343.76,256) and (346,258.24) .. (346,261) .. controls (346,263.76) and (343.76,266) .. (341,266) .. controls (338.24,266) and (336,263.76) .. (336,261) -- cycle ;
%Shape: Circle [id:dp33398840599061985] 
\draw  [fill={rgb, 255:red, 208; green, 132; blue, 223 }  ,fill opacity=1 ] (356,261) .. controls (356,258.24) and (358.24,256) .. (361,256) .. controls (363.76,256) and (366,258.24) .. (366,261) .. controls (366,263.76) and (363.76,266) .. (361,266) .. controls (358.24,266) and (356,263.76) .. (356,261) -- cycle ;
%Shape: Circle [id:dp16293493728877362] 
\draw  [fill={rgb, 255:red, 208; green, 132; blue, 223 }  ,fill opacity=1 ] (376,261) .. controls (376,258.24) and (378.24,256) .. (381,256) .. controls (383.76,256) and (386,258.24) .. (386,261) .. controls (386,263.76) and (383.76,266) .. (381,266) .. controls (378.24,266) and (376,263.76) .. (376,261) -- cycle ;
%Shape: Circle [id:dp6475416391388906] 
\draw  [fill={rgb, 255:red, 208; green, 132; blue, 223 }  ,fill opacity=1 ] (396,261) .. controls (396,258.24) and (398.24,256) .. (401,256) .. controls (403.76,256) and (406,258.24) .. (406,261) .. controls (406,263.76) and (403.76,266) .. (401,266) .. controls (398.24,266) and (396,263.76) .. (396,261) -- cycle ;
%Shape: Circle [id:dp9294871350945758] 
\draw  [fill={rgb, 255:red, 208; green, 132; blue, 223 }  ,fill opacity=1 ] (416,261) .. controls (416,258.24) and (418.24,256) .. (421,256) .. controls (423.76,256) and (426,258.24) .. (426,261) .. controls (426,263.76) and (423.76,266) .. (421,266) .. controls (418.24,266) and (416,263.76) .. (416,261) -- cycle ;
%Shape: Circle [id:dp2516604947925489] 
\draw  [fill={rgb, 255:red, 208; green, 132; blue, 223 }  ,fill opacity=1 ] (436,261) .. controls (436,258.24) and (438.24,256) .. (441,256) .. controls (443.76,256) and (446,258.24) .. (446,261) .. controls (446,263.76) and (443.76,266) .. (441,266) .. controls (438.24,266) and (436,263.76) .. (436,261) -- cycle ;
%Shape: Circle [id:dp6077178791848663] 
\draw  [fill={rgb, 255:red, 208; green, 132; blue, 223 }  ,fill opacity=1 ] (456,261) .. controls (456,258.24) and (458.24,256) .. (461,256) .. controls (463.76,256) and (466,258.24) .. (466,261) .. controls (466,263.76) and (463.76,266) .. (461,266) .. controls (458.24,266) and (456,263.76) .. (456,261) -- cycle ;
%Shape: Circle [id:dp10162037304394733] 
\draw  [fill={rgb, 255:red, 208; green, 132; blue, 223 }  ,fill opacity=1 ] (476,261) .. controls (476,258.24) and (478.24,256) .. (481,256) .. controls (483.76,256) and (486,258.24) .. (486,261) .. controls (486,263.76) and (483.76,266) .. (481,266) .. controls (478.24,266) and (476,263.76) .. (476,261) -- cycle ;
%Shape: Circle [id:dp4353216307911423] 
\draw  [fill={rgb, 255:red, 208; green, 132; blue, 223 }  ,fill opacity=1 ] (496,261) .. controls (496,258.24) and (498.24,256) .. (501,256) .. controls (503.76,256) and (506,258.24) .. (506,261) .. controls (506,263.76) and (503.76,266) .. (501,266) .. controls (498.24,266) and (496,263.76) .. (496,261) -- cycle ;
%Shape: Circle [id:dp8639629333221042] 
\draw  [fill={rgb, 255:red, 208; green, 132; blue, 223 }  ,fill opacity=1 ] (516,261) .. controls (516,258.24) and (518.24,256) .. (521,256) .. controls (523.76,256) and (526,258.24) .. (526,261) .. controls (526,263.76) and (523.76,266) .. (521,266) .. controls (518.24,266) and (516,263.76) .. (516,261) -- cycle ;
%Shape: Circle [id:dp9380225032183107] 
\draw  [fill={rgb, 255:red, 208; green, 132; blue, 223 }  ,fill opacity=1 ] (536,261) .. controls (536,258.24) and (538.24,256) .. (541,256) .. controls (543.76,256) and (546,258.24) .. (546,261) .. controls (546,263.76) and (543.76,266) .. (541,266) .. controls (538.24,266) and (536,263.76) .. (536,261) -- cycle ;
%Shape: Circle [id:dp69532839022585] 
\draw  [fill={rgb, 255:red, 208; green, 132; blue, 223 }  ,fill opacity=1 ] (556,261) .. controls (556,258.24) and (558.24,256) .. (561,256) .. controls (563.76,256) and (566,258.24) .. (566,261) .. controls (566,263.76) and (563.76,266) .. (561,266) .. controls (558.24,266) and (556,263.76) .. (556,261) -- cycle ;
%Shape: Circle [id:dp7123066685399877] 
\draw  [fill={rgb, 255:red, 208; green, 132; blue, 223 }  ,fill opacity=1 ] (576,261) .. controls (576,258.24) and (578.24,256) .. (581,256) .. controls (583.76,256) and (586,258.24) .. (586,261) .. controls (586,263.76) and (583.76,266) .. (581,266) .. controls (578.24,266) and (576,263.76) .. (576,261) -- cycle ;
%Shape: Circle [id:dp16230619332726415] 
\draw  [fill={rgb, 255:red, 208; green, 132; blue, 223 }  ,fill opacity=1 ] (596,261) .. controls (596,258.24) and (598.24,256) .. (601,256) .. controls (603.76,256) and (606,258.24) .. (606,261) .. controls (606,263.76) and (603.76,266) .. (601,266) .. controls (598.24,266) and (596,263.76) .. (596,261) -- cycle ;
%Shape: Circle [id:dp5821600203536024] 
\draw  [fill={rgb, 255:red, 208; green, 132; blue, 223 }  ,fill opacity=1 ] (616,261) .. controls (616,258.24) and (618.24,256) .. (621,256) .. controls (623.76,256) and (626,258.24) .. (626,261) .. controls (626,263.76) and (623.76,266) .. (621,266) .. controls (618.24,266) and (616,263.76) .. (616,261) -- cycle ;
%Straight Lines [id:da322212312945216] 
\draw [color={rgb, 255:red, 139; green, 87; blue, 42 }  ,draw opacity=1 ][line width=2.25]    (341,81) -- (477.04,100.43) ;
\draw [shift={(481,101)}, rotate = 188.13] [color={rgb, 255:red, 139; green, 87; blue, 42 }  ,draw opacity=1 ][line width=2.25]    (17.49,-5.26) .. controls (11.12,-2.23) and (5.29,-0.48) .. (0,0) .. controls (5.29,0.48) and (11.12,2.23) .. (17.49,5.26)   ;
%Straight Lines [id:da4489352767724074] 
\draw [color={rgb, 255:red, 121; green, 152; blue, 89 }  ,draw opacity=1 ][line width=2.25]    (181,121) -- (257.12,140.03) ;
\draw [shift={(261,141)}, rotate = 194.04] [color={rgb, 255:red, 121; green, 152; blue, 89 }  ,draw opacity=1 ][line width=2.25]    (17.49,-5.26) .. controls (11.12,-2.23) and (5.29,-0.48) .. (0,0) .. controls (5.29,0.48) and (11.12,2.23) .. (17.49,5.26)   ;
%Straight Lines [id:da29293541237256127] 
\draw [color={rgb, 255:red, 245; green, 166; blue, 35 }  ,draw opacity=1 ][line width=2.25]    (181,61) -- (217.42,79.21) ;
\draw [shift={(221,81)}, rotate = 206.57] [color={rgb, 255:red, 245; green, 166; blue, 35 }  ,draw opacity=1 ][line width=2.25]    (17.49,-5.26) .. controls (11.12,-2.23) and (5.29,-0.48) .. (0,0) .. controls (5.29,0.48) and (11.12,2.23) .. (17.49,5.26)   ;
%Straight Lines [id:da6503841605853918] 
\draw [line width=2.25]    (301,141) -- (282.26,84.79) ;
\draw [shift={(281,81)}, rotate = 71.57] [color={rgb, 255:red, 0; green, 0; blue, 0 }  ][line width=2.25]    (17.49,-5.26) .. controls (11.12,-2.23) and (5.29,-0.48) .. (0,0) .. controls (5.29,0.48) and (11.12,2.23) .. (17.49,5.26)   ;
%Straight Lines [id:da7052965694682009] 
\draw [color={rgb, 255:red, 56; green, 163; blue, 139 }  ,draw opacity=1 ][line width=2.25]    (321,241) -- (457.04,260.43) ;
\draw [shift={(461,261)}, rotate = 188.13] [color={rgb, 255:red, 56; green, 163; blue, 139 }  ,draw opacity=1 ][line width=2.25]    (17.49,-5.26) .. controls (11.12,-2.23) and (5.29,-0.48) .. (0,0) .. controls (5.29,0.48) and (11.12,2.23) .. (17.49,5.26)   ;

% Text Node
\draw (83.67,59.33) node [anchor=north west][inner sep=0.75pt]  [font=\large,color={rgb, 255:red, 171; green, 105; blue, 184 }  ,opacity=1 ] [align=left] {$ $$\displaystyle r$};
% Text Node
\draw (527,204) node [anchor=north west][inner sep=0.75pt]  [color={rgb, 255:red, 208; green, 2; blue, 27 }  ,opacity=1 ] [align=left] {$\displaystyle \rho $};
% Text Node
\draw (23,164) node [anchor=north west][inner sep=0.75pt]  [font=\large] [align=left] {$\displaystyle \mu $};
% Text Node
\draw (83,204) node [anchor=north west][inner sep=0.75pt]  [font=\large,color={rgb, 255:red, 245; green, 166; blue, 35 }  ,opacity=1 ] [align=left] {$\displaystyle \nu _{1}$};
% Text Node
\draw (155.8,162.8) node [anchor=north west][inner sep=0.75pt]  [font=\large,color={rgb, 255:red, 121; green, 152; blue, 89 }  ,opacity=1 ] [align=left] {$\displaystyle \nu _{2}$};
% Text Node
\draw (278.78,181.11) node [anchor=north west][inner sep=0.75pt]  [font=\large,color={rgb, 255:red, 139; green, 87; blue, 42 }  ,opacity=1 ] [align=left] {$\displaystyle \nu _{3}$};
% Text Node
\draw (226,50) node [anchor=north west][inner sep=0.75pt]  [font=\large,color={rgb, 255:red, 245; green, 166; blue, 35 }  ,opacity=1 ] [align=left] {$\displaystyle \nu _{1}$};
% Text Node
\draw (294,58) node [anchor=north west][inner sep=0.75pt]  [font=\large] [align=left] {$\displaystyle \mu $};
% Text Node
\draw (243,104) node [anchor=north west][inner sep=0.75pt]  [font=\large,color={rgb, 255:red, 121; green, 152; blue, 89 }  ,opacity=1 ] [align=left] {$\displaystyle \nu _{2}$};
% Text Node
\draw (463,64) node [anchor=north west][inner sep=0.75pt]  [font=\large,color={rgb, 255:red, 139; green, 87; blue, 42 }  ,opacity=1 ] [align=left] {$\displaystyle \nu _{3}$};
% Text Node
\draw (463,124) node [anchor=north west][inner sep=0.75pt]  [font=\large,color={rgb, 255:red, 56; green, 163; blue, 139 }  ,opacity=1 ] [align=left] {$\displaystyle \nu _{4}$};

\end{tikzpicture}

%% file: main.bbl
\begin{thebibliography}{}
%
\bibitem{AFKKZ} I.~Arzhantsev, H.~Flenner, S.~Kaliman, F.~Kutzschebauch, and M.~Zaidenberg.
Flexible varieties and automorphism groups. Duke Mathematical Journal 162 (2013), 767--823.
%
\bibitem{AKZ} I. Arzhantsev, K. Kuyumzhiyan, M. Zaidenberg. Infinite transitivity,
finite generation, and Demazure roots. Advances in Mathematics 351 (2019), 1-32.
%
\bibitem{Co} D.~A.~Cox,
J.~B.~Little and
H.~K.~Schenck. Toric Varieties.  Graduate Studies in Mathematics, 124. American Mathematical Society., Providence, RI, 2011.
%
\bibitem{De}M.~Demazure. Sous-groupes algébriques de rang maximum du groupe de Cremona. Annales Scientifiques de l'École Normale Supérieure (4) 3 (1970), 507–588.
%
\bibitem{Fu} W.~Fulton. Introduction to Toric Varieties. Princeton University
Press, 1993.
%
\bibitem{Kr} J.-P.~Furter and H.~Kraft. On the geometry of the automorphism group of
affine n-space. arXiv:1809.04175, 2018.
%
\bibitem{Sasha} S.~Kovalenko, A.~Perepechko, and M.~Zaidenberg.
On automorphism groups of affine surfaces. In:
Algebraic varieties and automorphism groups.
Advanced Studies in Pure Mathematics 75, pages 207--286.
%
\bibitem{LPS} D.~Lewis, K.~Perry, and A.~Straub. An algorithmic approach to the 
Polydegree Conjecture for plane polynomial automorphisms. arXiv:1809.09681
%
\bibitem{Lie}
A.~Liendo. Affine T-varieties of complexity one and locally nilpotent derivations. Transformation Groups 15 (2010), 389–425.
%
\bibitem{Li}
A.~Liendo. $\mathbb{G}_a$-actions of fiber type on affine $T$-varieties. Journal of Algebra 324 (2010), 3653--3665.
%
\bibitem{LiC1}
Alvaro Liendo. Affine $\mathbb{T}$-varieties of complexity one and locally nilpotent derivations. Transform. Groups 15 (2010), no. 2, 398-425
\bibitem{Ma} M.~Manetti. The Baker-Campbell-Hausdorff formula. Notes of a course on deformation theory 2011-12. 
http://www1.mat.uniroma1.it/people/manetti/DT2011/BCHformula.pdf.

\end{thebibliography}
